\documentclass[a4paper, reqno, 11pt]{amsart}

\usepackage{color}
\usepackage{amsmath}
\usepackage{amsfonts}
\usepackage{amsthm}
\usepackage{graphicx}
\usepackage{eucal}
\usepackage{amscd}
\usepackage[all,2cell]{xy}
\usepackage{amssymb}
\usepackage{mathrsfs}

 \usepackage{tikz}
 \usetikzlibrary{positioning}
 \tikzset{mynode/.style={draw,circle,inner sep=1pt,outer sep=0pt}}

\theoremstyle{plain}
\newtheorem{teo}[subsection]{Theorem}
\newtheorem{cor}[subsection]{Corollary}
\newtheorem{lem}[subsection]{Lemma}
\newtheorem{prop}[subsection]{Proposition}

\theoremstyle{remark}
\newtheorem{defi}[subsection]{Definition}
\newtheorem{example}[subsection]{Example}

\newtheorem{remark}[subsection]{Remark}

\def\mathsfdef#1{\expandafter\def\csname#1\endcsname{{\sf#1}}}
\mathsfdef{Cat}
\mathsfdef{Ord}
\mathsfdef{ArrCAT}

\mathsfdef{a}
\mathsfdef{b}

\newdir{ |>}{{}*!/-3.5pt/@{|}*!/-8pt/:(1,-.2)@^{>}*!/-8pt/:(1,+.2)@_{>}}

\dedicatory{}

\begin{document}

\title{EFFECTIVE DESCENT MORPHISMS OF FILTERED PREORDERS}

\author{Maria Manuel Clementino}
\address[Maria Manuel Clementino]{CMUC and Departamento de
Matem\'atica, Universidade de Coimbra, 3000--143 Coimbra,
Portugal}
\thanks{
Partially supported by the Centre for Mathematics of
the University of Coimbra -- UIDB/00324/2020,
funded by the Portuguese Government through FCT/MCTES}
\email{mmc@mat.uc.pt}

\author{George Janelidze}
\address[George Janelidze]{Department of Mathematics and Applied Mathematics, University of Cape Town, Rondebosch 7700, South Africa}
\thanks{}
\email{george.janelidze@uct.ac.za}



\begin{abstract}

We characterize effective descent morphisms of what we call filtered preorders, and apply these results to slightly improve a known result, due to the first author and F. Lucatelli Nunes, on the effective descent morphisms in lax comma categories of preorders. A filtered preorder, over a fixed preorder $X$, is defined as a preorder $A$ equipped with a profunctor $X\to A$ and, equivalently, as a set $A$ equipped with a family $(A_x)_{x\in X}$ of upclosed subsets of $A$ with $x'\leqslant x\Rightarrow A_x\subseteq A_{x'}$.

\end{abstract}

\date{\today}

\maketitle

\section{Introduction}

Throughout this paper, by a \textit{preorder} we mean a set equipped with a reflexive transitive relation, and we write $X=(X,\leqslant)$ for a fixed preorder. Such preorders were simply called ordered sets, e.g. in \cite{[CL2023]}, and keeping the notation of \cite{[CL2023]}, we will write:
\begin{itemize}
	\item $\Ord$ for the category of preorders;
	\item $\Ord//X$ for the lax comma category of $\Ord$ over $X$, and we recall that its objects are the same as in $\Ord/X=(\Ord\downarrow X)$, while a morphism
\[f:(A,\alpha)\to(B,\beta)\] in $\Ord//X$ is a morphism $f:A\to B$ in $\Ord$ with $\alpha\leqslant\beta f$, that is, with \[\alpha(a)\leqslant\beta f(a)\] for all $a\in A$.
\end{itemize}
Considering one of the problems formulated in \cite{[CL2023]}, namely the problem of characterizing effective descent morphisms in $\Ord//X$, we found a `nicer' category $\Ord_X$, of what we called `filtered preorders', where this problem turned out to be easy enough to solve it completely. It also turned out that dealing with $\Ord_X$ helps to improve what was done in \cite{[CL2023]} with $\Ord//X$. Here, a filtered preorder, over a fixed preorder $X$, is defined as a preorder $A$ equipped with a profunctor $X\to A$ and, equivalently, as a set $A$ equipped with a family $(A_x)_{x\in X}$ of upclosed subsets of $A$ with $x'\leqslant x\Rightarrow A_x\subseteq A_{x'}$.

The paper is organized as follows:

Section 2 is devoted to general categorical remarks, notably involving the notion of map of adjunctions which was briefly mentioned in \cite{[M1971]}. These remarks should be reformulated in the style of \cite{[S1972]} one day, but we did not go further than what we needed for our purposes. We should particularly mention the contrast between `strict' and `pseudo-', which appears as the contrast between Corollary 9.6 in \cite{[L2018]} and our Corollary \ref{cor:p1p2}(b). We are satisfied with `strict' since we are using it for functors that preserve relevant chosen limits and colimits (specifically, pullbacks and coequalizers).

Section 3 introduces filtered preorders and shows how their category extends the category $\Ord//X$, while the purpose of Section 4 is to characterize effective descent morphisms in $\Ord_X$, and Section 5 gives a new class of effective descent morphisms in $\Ord//X$ slightly improving Theorem 3.9 of \cite{[CL2023]}.

\section{General remarks on monadicity and descent}

In this section we first deal with a fixed \emph{map $(P,P')\colon\Phi\to X$ of adjunctions} in the sense of \cite{[M1971]} (see Section 7 of Chapter IV therein) displayed as \[\xymatrix{\mathcal{U}\ar[d]_P\ar[r]^\Phi&\mathcal{U}'\ar[d]^{P'}\\\mathcal{W}\ar[r]_X&\mathcal{W}'}\] In this diagram $P$ and $P'$ are functors while
\[\Phi=(\Phi_!,\Phi^*,\eta^\Phi,\varepsilon^\Phi)\colon\mathcal{U}\to\mathcal{U}'\,\,\,\mathrm{and}\,\,\, X=(X_!,X^*,\eta^X,\varepsilon^X)\colon\mathcal{W}\to\mathcal{W}'\]
are adjunctions with \[X_!P=P'\Phi_!,\,\,P\Phi^*=X^*P',\,\,P\eta^\Phi=\eta^XP,\,\,P'\varepsilon^\Phi=\varepsilon^XP'.\]
Equivalently, we can present it as an adjunction \[((\Phi_!,X_!),(\Phi^*,X^*),(\eta^\Phi,\eta^X),(\varepsilon^\Phi,\varepsilon^X))\colon (\mathcal{U},P,\mathcal{W})\to(\mathcal{U}',P',\mathcal{W}'),\] in the 2-category $\ArrCAT$ ($=\Cat^\mathbf{2}$ in the notation of, e.g., \cite{[LS2002]}) of arrows of the category of categories. In order to make this clear, let us just recall that:
\begin{itemize}
	\item the objects of $\ArrCAT$ are all functors $F\colon \mathcal{A}\to\mathcal{B}$, written as triples $(\mathcal{A},F,\mathcal{B})$;
	\item a morphism $(\mathcal{A},F,\mathcal{B})\to(\mathcal{A}',F',\mathcal{B}')$ is a pair of functors $(K,L)$ making the diagram \[\xymatrix{\mathcal{A}\ar[d]_F\ar[r]^K&\mathcal{A}'\ar[d]^{F'}\\\mathcal{B}\ar[r]_L&\mathcal{B}'}\] commute;
	\item for morphisms $(K,L),(M,N)\colon (\mathcal{A},F,\mathcal{B})\to(\mathcal{A}',F',\mathcal{B}')$, a 2-cell from $(K,L)$ to $(M,N)$ is a pair of natural transformations $(\sigma,\tau)=$
	\[\xymatrix{\mathcal{A}\ar[d]_F\ar@/^2.5mm/[rr]^-K\ar@/_2mm/[rr]^{\downarrow\sigma}_-M&&\mathcal{A}'\ar[d]^{F'}\\
\mathcal{B}\ar@/^2.5mm/[rr]^-L\ar@/_2mm/[rr]^{\downarrow\tau}_-N&&\mathcal{B}'}\] with $F'\sigma=\tau F$.
\end{itemize}

Our first remark is: the equations $P\eta^\Phi=\eta^XP$ and $P'\varepsilon^\Phi=\varepsilon^XP'$ immediately imply:

\begin{lem}\label{lem:21}
	For the map of adjunctions above:
	\begin{enumerate}
		\item if $\eta^\Phi$ is an isomorphism and $P$ is surjective on objects, then $\eta^X$ is an isomorphism;
		\item if $\varepsilon^\Phi$ is an isomorphism and $P'$ is surjective on objects, then $\varepsilon^X$ is an isomorphism;
		\item in particular, if $\Phi$ is a category equivalence and both $P$ and $P'$ are surjective on objects, then $X$ is a category equivalence.\qed
	\end{enumerate}
\end{lem}

Next, our map of adjunctions above determines, whenever the categories $\mathcal{U}'$ and $\mathcal{W}'$ have chosen coequalizers preserved by the functor $P'$, its \textit{derived map of adjunctions}
\[\xymatrix{\mathcal{U}^{T^\phi}\ar[d]_{\tilde{P}}\ar[r]&\mathcal{U}'\ar[d]^{P'}\\\mathcal{W}^{T^X}\ar[r]&\mathcal{W}'}\]
where: $T^\phi$ and $T^X$ are the monads determined by the adjunctions $\Phi$ and $X$, respectively; $\tilde{P}$ is induced by $(P,P')$; and the horizontal arrows are the comparison adjunctions. If $(P,P')$ is a split epimorphism of adjunctions, then $\tilde{P}$ also is a split epimorphism and, using Lemma \ref{lem:21}, one can easily show that monadicity of $\Phi$ implies monadicity of $X$. The precise statement is:

\begin{lem}\label{lem:22}
	Let $(P,P')\colon \Phi\to X$ and $(I,I')\colon X\to\Phi$ be maps of adjunctions with $PI=1_{\mathcal{W}}$ and $P'I'=1_{\mathcal{W}'}$. If the categories $\mathcal{U}'$ and $\mathcal{W}'$ have chosen coequalizers preserved by the functors $P'$ and $I'$, then premonadicity of $\Phi$ implies premonadicity of $X$ and monadicity of $\Phi$ implies monadicity of $X$.\qed
\end{lem}

Let $\mathcal{C}_0$ and $\mathcal{C}_1$ be categories with chosen pullbacks, and $S_1\colon \mathcal{C}_1\to\mathcal{C}_0$ a functor that preserves them. Then any morphism $p_1\colon E_1\to B_1$ in $\mathcal{C}_1$ determines a map
\[\xymatrix{(\mathcal{C}_1\downarrow E_1)\ar[d]_P\ar[rrr]^-{(p_{1!},p_1^*,\eta^{p_1},\varepsilon^{p_1})}&&&(\mathcal{C}_1\downarrow B_1)\ar[d]^{P'}\\(\mathcal{C}_0\downarrow E_0)\ar[rrr]_-{(p_{0!},p_0^*,\eta^{p_0},\varepsilon^{p_0})}&&&(\mathcal{C}_0\downarrow B_0)}\]
of adjunctions, in which $(p_0\colon E_0\to B_0)=S_1(p_1\colon E_1\to B_1)$, $P$ and $P'$ are induced by $S_1$, and the horizontal arrows are the suitable change-of-base adjunctions. From Lemma \ref{lem:22}, having in mind the monadic approach to descent theory (see e.g. \cite{[JT1994]}) we obtain:

\begin{cor}\label{cor:eff}
	Let $\mathcal{C}_0$ and $\mathcal{C}_1$ be categories with chosen pullbacks and chosen coequalizers, and $S_1\colon \mathcal{C}_1\to\mathcal{C}_0$ and $J_1\colon \mathcal{C}_0\to\mathcal{C}_1$ functors that preserves them and have $S_1J_1=1_{\mathcal{C}_0}$. Then the functor $S_1$ sends descent morphisms in $\mathcal{C}_1$ to descent morphisms in $\mathcal{C}_0$ and effective descent morphisms in $\mathcal{C}_1$ to effective descent morphisms in $\mathcal{C}_0$.\qed
\end{cor}

Now consider a cube diagram
\[\xymatrix@C=20pt@R=20pt{&\mathcal{V}\ar[ddd]^(0.6){Q}\ar[rrrr]^-\Psi&&&&\mathcal{V'}\ar[ddd]^(0.6){Q'}\\
\mathcal{U}\times_{\mathcal{W}}\mathcal{V}
\ar[ur]^{\Pi_2}\ar[rrrr]^-{\Phi\times_X\Psi}\ar[ddd]_(0.4){\Pi_1}&&&&\mathcal{U}'\times_{\mathcal{W}'}\mathcal{V}'
\ar[ddd]_(0.4){\Pi'_1}\ar[ur]^{\Pi'_2}\\\\&\mathcal{W}\ar[rrrr]_X&&&&\mathcal{W'}\\\mathcal{U}\ar[ur]_P\ar[rrrr]_-\Phi&&&&\mathcal{U}'\ar[ur]_{P'}}\]
whose left-hand and the right-hand faces are pullbacks in the category of categories, while all other faces are maps of adjunctions.
Here we assume that  $\mathcal{U}\times_{\mathcal{W}}\mathcal{V}$ and $\mathcal{U}'\times_{\mathcal{W}'}\mathcal{V}'$ are constructed `as usually', that is, as suitable categories of pairs, and the adjunction $\Phi\times_X\Psi$ is defined by
\[(\Phi\times_X\Psi)_!(U,V)=(\Phi_!(U),\Psi_!(V)),\,\,(\Phi\times_X\Psi)^*(U',V')=(\Phi^*(U'),\Psi^*(V')),\]
\[\eta^{\Phi\times_X\Psi}_{(U,V)}=(\eta^\Phi_U,\eta^\Psi_V),\,\,\varepsilon^{\Phi\times_X\Psi}_{(U',V')}=
(\varepsilon^\Phi_{U'},\varepsilon^\Psi_{V'}).\]

From the construction of $\eta^{\Phi\times_X\Psi}_{(U,V)}$ and $\varepsilon^{\Phi\times_X\Psi}_{(U',V')}$, we obtain:
\begin{lem}\label{lem:equiv} Under the assumptions above:
	\begin{enumerate}
		\item if $\eta^\Phi$ and $\eta^\Psi$ are isomorphisms, then so is $\eta^{\Phi\times_X\Psi}$;
		\item if $\varepsilon^\Phi$ and $\varepsilon^\Psi$ are isomorphisms, then so is $\varepsilon^{\Phi\times_X\Psi}$;
		\item if $\Phi$ and $\Psi$ are category equivalences, then so is $\Phi\times_X\Psi$.\qed
	\end{enumerate}
\end{lem}

If the categories $\mathcal{U'}$, $\mathcal{V'}$, and $\mathcal{W'}$ have chosen coequalizers preserved by the functors $P'$ and $Q'$, then we have the \textit{derived cube diagram}	
\[\xymatrix@C=20pt@R=20pt{&\mathcal{V}^{T^\Psi}\ar[ddd]^(0.6){\tilde{Q}}\ar[rrrr]&&&&\mathcal{V'}\ar[ddd]^(0.6){Q'}\\(\mathcal{U}
\times_{\mathcal{W}}\mathcal{V})^{T^{\Phi\times_X\Psi}}\ar[ur]^{\tilde{\Pi}_2}\ar[rrrr]
\ar[ddd]_(0.4){\tilde{\Pi}_1}&&&&\mathcal{U}'\times_{\mathcal{W}'}\mathcal{V}'\ar[ddd]_(0.4){\Pi'_1}
\ar[ur]^{\Pi'_2}\\\\&\mathcal{W}^{T^X}\ar[rrrr]&&&&\mathcal{W'}\\\mathcal{U}^{T^\phi}\ar[ur]_{\tilde{P}}
\ar[rrrr]&&&&\mathcal{U}'\ar[ur]_{P'}}\]
obtained in an obvious way using derived maps of adjunctions. Note, in particular, that $(\mathcal{U}\times_{\mathcal{W}}\mathcal{V})^{T^{\Phi\times_X\Psi}}$ can be identified with $\mathcal{U}^{T^\phi}\times_{\mathcal{W}^{T^X}}\mathcal{V}^{T^{\Psi}}$. And, from Lemma \ref{lem:equiv}, we obtain

\begin{lem}\label{lem:25}
	For maps $(P,P')\colon \Phi\to X$ and $(Q,Q')\colon \Psi\to X$ of adjunctions, premonadicity of $\Phi$ and $\Psi$ implies premonadicity of $\Phi\times_X\Psi$ and monadicity of $\Phi$ and $\Psi$ implies premonadicity of $\Phi\times_X\Psi$, provided the categories $\mathcal{U}'$, $\mathcal{V}'$, and $\mathcal{W}'$ have chosen coequalizers preserved by the functors $P'$ and $Q'$. \qed
\end{lem}

Let
\[\xymatrix{\mathcal{C}_1\ar[r]^-{S_1}&\mathcal{C}_0&\mathcal{C}_2\ar[l]_-{S_2},}\]
be a cospan of categories having chosen pullbacks preserved by the functors $S_1$ and $S_2$. Given a morphism $(p_1,p_2)\colon (E_1,E_2)\to(B_1,B_2)$ in $\mathcal{C}_1\times_{\mathcal{C}_0}\mathcal{C}_2$, we can take the originally considered cube diagram to be
\[\xymatrix@C=20pt@R=20pt{&(\mathcal{C}_2\downarrow E_2)\ar[ddd]^{Q}\ar[rr]&&(\mathcal{C}_2\downarrow B_2)\ar[ddd]^{Q'}\\((\mathcal{C}_1\times_{\mathcal{C}_0}\mathcal{C}_2)\downarrow(E_1,E_2))\ar[ur]^{\Pi_2}\ar[rr]\ar[ddd]_{\Pi_1}&&((\mathcal{C}_1\times_{\mathcal{C}_0}\mathcal{C}_2)\downarrow(B_1,B_2))\ar[ddd]_-{\Pi'_1}\ar[ur]^{\Pi'_2}\\\\&(\mathcal{C}_0\downarrow E_0)\ar[rr]&&(\mathcal{C}_0\downarrow B_0)\\(\mathcal{C}_1\downarrow E_1)\ar[ur]_P\ar[rr]&&(\mathcal{C}_1\downarrow B_1)\ar[ur]_{P'}}\]
where $(P,P')$ is induced by $S_1$, $(Q,Q')$ is induced by $S_2$, and horizontal arrows are the suitable change-of-base adjunctions. From Lemma \ref{lem:25}, since the change-of-base adjunction along a morphism is premonadic/monadic if and only if it is a descent/effective descent morphism (see e.g. \cite{[JT1994]}), we obtain:

\begin{cor}\label{cor:p1p2}
	If the categories $\mathcal{C}_i$ $(i=0,1,2)$ have chosen pullbacks and chosen coequalizers preserved by the functors $S_1$ and $S_2$, then:
	\begin{enumerate}
		\item if $p_1$ and $p_2$ are descent morphisms, then so is $(p_1,p_2)$;
		\item if $p_1$ and $p_2$ are effective descent morphisms, then so is $(p_1,p_2)$.\qed
	\end{enumerate}
\end{cor}

\begin{example}\label{exa:27}
Let us choose the data above as follows:
	\begin{itemize}
		\item $\mathcal{C}_0$ to be the category of sets.
		\item $\mathcal{C}_1=\Ord$.
		\item $\mathcal{C}_2$ to be the category of pairs $(A,(A_x)_{x\in X})$, where $A$ is a set and $(A_x)_{x\in X}$ is a family of subsets of $A$ with $x'\leqslant x\Rightarrow A_x\subseteq A_{x'}$. A morphism \[f\colon (A,(A_x)_{x\in X})\to(B,(B_x)_{x\in X})\] in $\mathcal{C}_2$ is a morphism $f\colon A\to B$ in $\mathcal{C}_0$ with $f(A_x)\subseteq B_x$ for each $x\in X$.
		\item $S_1$ and $S_2$ to be the underlying set functors. Note that in all our categories here the pullbacks and coequalizers are \textit{chosen as for sets} and preserved by the functors $S_1$ and $S_2$; moreover, one could easily choose suitable right inverses of $S_1$ and $S_2$.
	\end{itemize}
We can write the objects of $\mathcal{C}_1\times_{\mathcal{C}_0}\mathcal{C}_2$ simply as $A=(A,(A_x)_{x\in X})$ assuming that $A$ is a set equipped with a preorder relation and a family $(A_x)_{x\in X}$ of subsets satisfying the condition above. A morphism $p\colon E\to B$ in $\mathcal{C}_1\times_{\mathcal{C}_0}\mathcal{C}_2$ is then a map $p\colon E\to B$ that preserves the preorder relation and has $p(E_x)\subseteq B_x$ for each $x\in X$. It is easy to show that $p$ is an effective descent morphism in $\mathcal{C}_2$ if and only if $p$ and all induced maps $p_x\colon E_x\to B_x$ are surjective. Then, putting this together with the description of effective descent morphisms in $\Ord$ in Proposition 3.4 of \cite{[JS2002]}, and applying Corollaries \ref{cor:eff} and \ref{cor:p1p2}, we conclude that

\begin{teo}
A morphism $p\colon E\to B$ in $\mathcal{C}_1\times_{\mathcal{C}_0}\mathcal{C}_2$ is an effective descent morphism if and only if:
	\begin{itemize}
		\item [(a)] for every $b_2\leqslant b_1\leqslant b_0$ in $B$, there exist $e_2\leqslant e_1\leqslant e_0$ in $E$ with $p(e_i)=b_i$ for each $i=0,1,2$;
		\item [(b)] $p$ and all induced maps $p_x\colon E_x\to B_x$ are surjective.
	\end{itemize}
\end{teo}
\end{example}

\section{$X$-filtered preorders}

\begin{defi}
	An \emph{$X$-filtered preorder} is a pair $(A,\a)$, where $\mathsf{a}\colon X\to A$ is a $\{0,1\}$-profunctor; that is, $\mathsf{a}\subseteq X\times A$ is a relation satisfying
\[(x'\leqslant x\,\,\&\,\,a\leqslant a') \Rightarrow ((x,a)\in\mathsf{a}\Rightarrow (x',a')\in\mathsf{a})\]
for all $x,x'\in X$ and $a,a'\in A$. A morphism $f\colon (A,\mathsf{a})\to(B,\mathsf{b})$ in the category $\Ord_X$ of $X$-filtered preorders is a morphism $f\colon A\to B$ in $\Ord$ with $f\mathsf{a}\leqslant\mathsf{b}$, that is, with \[(x,a)\in\mathsf{a}\Rightarrow(x,f(a))\in\mathsf{b}\] for all $x\in X$ and $a\in A$.
\end{defi}

The reason for calling such pairs $(A,\mathsf{a})$ $X$-filtered preorders is that a profunctor $\mathsf{a}\colon X\to A$ can be equivalently described as an $X$-\textit{filtration} on $A$ defined as a family $(A_x)_{x\in X}$ of upclosed subsets of $A$ with $x'\leqslant x\Rightarrow A_x\subseteq A_{x'}$. The relationship between these two types of structure is straightforward and well known at various levels of generality; it is given by \[A_x=\mathsf{a}(x,-)=\{a\in A\mid(x,a)\in\mathsf{a}\}\]
or, equivalently, by
\[\mathsf{a}=\{(x,a)\in X\times A\mid a\in A_x\}.\]
If there is no danger of confusion, we will simply write
\[(A,\mathsf{a})=A=(A,(A_x)_{x\in X})\] (the second equality here is what we have already used in a more general situation in Example \ref{exa:27}).

We could also similarly describe these structures as families of $A$-indexed (for various $A$) families of subsets
\[\mathsf{a}(-,a)=\{x\in X\mid(x,a)\in\mathsf{a}\}\] of $X$, but that would be less useful since we are considering a fixed $X$, not a fixed $A$.

We will use the fully faithful functors
\[\xymatrix{\Ord//X\ar[r]^-{F_1}&\Ord_X\ar[r]^-{F_2}&\mathcal{C}_1\times_{\mathcal{C}_0}\mathcal{C}_2}\]
where $F_1$ is defined by \[F_1(A,\alpha)=(A,\{(x,a)\mid x\leqslant\alpha(a)\})\] and by requiring the diagram
\[\xymatrix{\Ord//X\ar[dr]_{\mathrm{forgetful}}\ar[rr]^-{F_1}&&\Ord_X\ar[dl]^{\mathrm{forgetful}}\\&\Ord}\] to commute, $\mathcal{C}_1\times_{\mathcal{C}_0}\mathcal{C}_2$ is as in Example \ref{exa:27}, and $F_2$ is the inclusion functor. The following proposition is also well known at various levels of generality, but instead of explaining that we give a (simple) direct proof:

\begin{prop}\label{prop:32}
	The image of $F_1\colon \Ord//X\to\Ord_X$, which is the same as its replete image, consists of those filtered preorders $A$ in which, for every $a\in A$, the subset $\mathsf{a}(-,a)$ of $X$ has a largest element; equivalently, there is a largest $x\in X$ with $a\in A_x$.
\end{prop}
\begin{proof}
	For $(A,\alpha)$ and $(B,\beta)$ in $\Ord//X$, a morphism $f\colon A\to B$ in $\Ord$, and $a\in A$, we have
\[\alpha(a)\leqslant\beta f(a)\Leftrightarrow\forall_x(x\leqslant\alpha(a)\Rightarrow x\leqslant\beta f(a)),\]
	which means that $f$ is a morphism from $(A,\alpha)$ to $(B,\beta)$ in $\Ord//X$ if and only if it is a morphism from $F_1(A,\alpha)$ to $F_1(B,\beta)$ in $\Ord_X$. That is, $F_1$ is indeed fully faithful.
	
	Let $(A,\mathsf{a})$ be an object in $\Ord_X$ such that each $\mathsf{a}(-,a)$ has a largest element. Let $\alpha(a)$ be such a largest element (for each $a\in A$), and consider the so defined $(A,\alpha)$. We have
\[a\leqslant a'\Rightarrow\alpha(a)\in\a(-,a')\Rightarrow\alpha(a)\leqslant\alpha(a'),\]
	and so $(A,\alpha)$ is an object in $\Ord//X$. We also have $x\leqslant\alpha(a)\Leftrightarrow(x,a)\in\mathsf{a}$, and so $F_1(A,\alpha)=(A,\mathsf{a})$.
	
	Conversely, if $(A,\mathsf{a})=F_1(A,\alpha)=(A,\{(x,a)\mid x\leqslant\alpha(a)\})$, then, obviously, for each $a\in A$, $\alpha(a)$ is the largest $x\in X$ with $x\leqslant\alpha(a)$, which is equivalent to $x\in\mathsf{a}(-,a)$.
\end{proof}

\section{Effective descent morphisms in $\Ord_X$}

We will often use a pullback diagram
\[\xymatrix{E\times_BA\ar[d]_{\pi_1}\ar[r]^-{\pi_2}&A\ar[d]^f\\E\ar[r]_p&B}\] in a given category, and refer to it simply as the pullback of $p$ and $f$. Recall, e.g. from part 2 of Corollary 2.7 in \cite{[JT1994]}:

\begin{prop}\label{prop:obst}
	Let $\mathcal{D}$ be a category with pullbacks and $\mathcal{C}$ a full subcategory in $\mathcal{D}$ closed under pullbacks. For an effective descent morphism $p\colon E\to B$ in $\mathcal{D}$, the following conditions are equivalent:
	\begin{itemize}
		\item [(i)] $p$ is an effective descent morphism in $\mathcal{C}$;
		\item [(ii)] for every morphism $f\colon A\to B$ in $\mathcal{D}$, we have $E\times_BA\in\mathcal{C}\Rightarrow A\in\mathcal{C}$.\qed
	\end{itemize}
\end{prop}
Using this proposition applied to the inclusion $F_2\colon \Ord_X\to\mathcal{C}_1\times_{\mathcal{C}_0}\mathcal{C}_2$ and the results of Section 2, we will prove
\begin{teo}\label{teo:OrdX}
A morphism $p\colon E\to B$ in $\Ord_X$ is an effective descent morphism if and only if:
\begin{itemize}
	\item [(a)] for each $b_2\leqslant b_1\leqslant b_0$ in $B$, there exist $e_2\leqslant e_1\leqslant e_0$ in $E$ with $p(e_i)=b_i$ for each $i=0,1,2$;
	\item [(b)] for each $x\in X$ and each $b_1\leqslant b_0$ in $B_x$, there exist $e_1\leqslant e_0$ in $E_x$ with $p(e_i)=b_i$ for each $i=0,1$.
\end{itemize}
\end{teo}
\begin{proof}
	``Only if'': Suppose $p$ is an effective descent morphism in $\Ord_X$. In Corollary \ref{cor:eff}, take $S_1$ to be the forgetful functor $\Ord_X\to\Ord$ and $J_1\colon \Ord\to\Ord_X$ to be defined by \[J_1(A)=(A,(A_x)_{x\in X})\,\,\,\text{with}\,\,\,A_x=A \text{ for all }x\in X.\] According to the description of effective descent morphisms in $\Ord$, it follows that condition (a) is satisfied.
	
	To prove that condition (b) is also satisfied, take any $x\in X$ and $b_1\leqslant b_0$ in $B_x$, and consider the pullback of $p$ and $f$ in $\mathcal{C}_1\times_{\mathcal{C}_0}\mathcal{C}_2$ with $A=\{b_1,b_0\}$ having $b_i\leqslant b_j\Leftrightarrow j\leqslant i$, $f$ being the inclusion map, and
	\begin{equation*}
	A_{x'}=\left\{\begin{array}{ll}\{b_1,b_0\},&\text{ if}\ x'< x;\\\{b_1\},& \text{ if}\ x'\sim x;\\
\emptyset,& \text{ if}\ x'\nleqslant x,\end{array}\right.
	\end{equation*}
where by $x'\sim x$ we mean $x'\leqslant x$ and $x\leqslant x'$, and by $x'< x$ we mean $x'\leqslant x$ and $x'\not\sim x$.
	In this case \[E\times_BA=(p^{-1}\{b_1\}\times\{b_1\})\cup(p^{-1}\{b_0\}\times\{b_0\})\]
	with
	\begin{equation*}
	(E\times_BA)_{x'}=\left\{\begin{array}{ll}((E_{x'}\cap p^{-1}\{b_1\})\times\{b_1\})\cup((E_{x'}\cap p^{-1}\{b_0\})\times\{b_0\}),& \text{ if}\ x'< x;\\
(E_{x'}\cap p^{-1}\{b_1\})\times\{b_1\},&    \text{ if}\ x'\sim x;\\\emptyset,& \text{ if}\ x'\nleq x.\end{array}\right.
	\end{equation*}
Since $A_x$ is not upclosed, $A$ does not belong to $\Ord_X$. Hence, since $p$ is an effective descent morphism in $\Ord_X$, it follows, by Proposition \ref{prop:obst}, that $E\times_BA$ does not belong to $\Ord_X$.
Therefore, at least one $(E\times_BA)_{x'}$ is not upclosed in $E\times_BA$.
However, suppose $x'< x$. Then \[(E\times_BA)_{x'}=(E_{x'}\cap p^{-1}\{b_1\})\times\{b_1\})\cup((E_{x'}\cap p^{-1}\{b_0\})\times\{b_0\})\] and, since $E_{x'}$ is upclosed in $E$, it is upclosed in $E\times_BA$.
Or, suppose $x'\nleqslant x$. Then $(E\times_BA)_{x'}$ is empty and so it is upclosed in $E\times_BA$.
Hence upclosedness of $(E\times_BA)_{x'}$ must fail for some $x'\sim x$; that is, $(E\times_BA)_{x'}$, and therefore $(E\times_BA)_x$, is not upclosed in $E\times_BA$.
Since $E_x$ is upclosed in $E$, \[(E\times_BA)_{x}=(E_{x}\cap p^{-1}\{b_1\})\times\{b_1\}\] is upclosed in $p^{-1}\{b_1\}\times\{b_1\}$, and so
there exist $(e_1,b_1)\in(E_{x}\cap p^{-1}\{b_1\})\times\{b_1\}$ and $(e_0,b_0)\in p^{-1}\{b_0\}\times\{b_1\}$ with $(e_1,b_1)\leqslant(e_0,b_0)$. This gives
		$e_1\leqslant e_0$, both in $E_{x}$, since $e_1$ belongs to $E_{x}$ and $E_{x}$ is upclosed in $E$, with $p(e_1)=b_1$ and $p(e_0)=b_0.$

	``If'': Suppose conditions (a) and (b) hold. This makes $p$ an effective descent morphism in $\mathcal{C}_1\times_{\mathcal{C}_0}\mathcal{C}_2$ (see Example \ref{exa:27}). Therefore, according to Proposition \ref{prop:obst}, it suffices to prove that, for every morphism $f\colon A\to B$ in $\mathcal{C}_1\times_{\mathcal{C}_0}\mathcal{C}_2$ with each $(E\times_BA)_x$ upclosed, each $A_x$ is also upclosed.
	
	Suppose $a\leqslant a'$ in $A$ with $a\in A_x$. As follows from condition (b), there exist $e,e'\in E$ with $e\leqslant e'$, $p(e)=a$, and $p(e')=a'$. This gives $(e,a)\leqslant(e',a')$ in $E\times_BA$ with $(e,a)\in(E\times_BA)_x$. Since $(E\times_BA)_x$ is upclosed in $E\times_BA$, it follows that $(e',a')$ is in $(E\times_BA)_x$, and so $a'$ is in $A_x$, since the projection $\pi_2\colon E\times_BA\to A$ is a morphism in $\mathcal{C}_1\times_{\mathcal{C}_0}\mathcal{C}_2$.
\end{proof}

\section{Effective descent morphisms in $\Ord//X$}

In this section we assume that $X$ is \textit{locally complete}, in the sense that, for each $x\in X$, the preorder $\{x'\in X\mid x'\leqslant x\}$ is equivalent to a complete lattice. We will also identify the category $\Ord//X$ with its $F_1$-image in $\Ord_X$; note that $\Ord//X$ is then closed under pullbacks in $\Ord_X$, thanks to the local completeness. While working with the morphism $p\colon E\to B$ we will write $B=(B,\beta)$, $E=(E,\varepsilon)$, and $E\times_BA=(E\times_BA,\gamma)$ in the notation we used for $\Ord//X$.

\begin{lem}\label{lem:trans}
	Let $p\colon E\to B$ be a morphism in $\Ord//X$ such that each induced map \linebreak $p_x\colon E_x\to B_x$ $(x\in X)$ is surjective. If $f\colon A\to B$ is a morphism in $\Ord_X$ with $E\times_BA$ in $\Ord//X$, then $A$ is in $\Ord//X$.
\end{lem}

\begin{proof}
	According to Proposition \ref{prop:32}, we have to prove that, given $a\in A$,  there is a largest $x\in X$ with $a\in A_x$. Following the proof of Theorem 3.9 in \cite{[CL2023]}, we are going to put \[\alpha(a)=\bigvee\{x\in X\mid a\in A_x\}\] and prove that it is such an element. First we observe that,
if $a$ belongs to $A_x$, then $f(a)$ belongs $B_x$ and so $x\leqslant\beta f(a)$. Therefore
		\[\{x\in X\mid a\in A_x\}=\{x\leqslant\beta f(a)\mid a\in A_x\},\]
		and so the join above is well defined.
By definition of $\alpha(a)$, it suffices to prove that $a\in A_{\alpha(a)}$. As $f(a)\in B_{\alpha(a)}$,
we have $\alpha(a)\leqslant\beta f(a)$, and so $f(a)$ belongs to $B_{\alpha(a)}$.
By our hypotheses, there exist $e\in E_{\alpha(a)}$ with $p(e)=f(a)$. Hence, for every $x\leqslant\alpha(a)$,
$(e,a)$ belongs to $E_x\times_BA_x=(E\times_BA)_x$, and so $x\leqslant\gamma(e,a)$. Since $\alpha(a)$ is the join of such elements $x$, it follows that $\alpha(a)\leqslant\gamma(e,a)$.
Then $a=\pi_2(c,a)\in A_{\gamma(c,a)}\subseteq A_{\alpha(a)}$, which completes our proof.
\end{proof}

From Proposition \ref{prop:obst}, Theorem \ref{teo:OrdX}, and Lemma \ref{lem:trans}, we obtain:

\begin{teo}\label{teo:OrdXI}
	If $X$ is a locally complete preordered set, then a morphism $p\colon E\to B$ in $\Ord//X$ is effective for descent in $\Ord//X$ provided that:
\begin{itemize}
	\item [(a)] for each $b_2\leqslant b_1\leqslant b_0$ in $B$, there exist $e_2\leqslant e_1\leqslant e_0$ in $E$ with $p(e_i)=b_i$ for each $i=0,1,2$;
	\item [(b$'$)] for each $x\in X$ and each $b_1\leqslant b_0$ with $x\leqslant \beta(b_1)$, there exist $e_1\leqslant e_0$ with $x\leqslant \varepsilon(e_1)$ and $p(e_i)=b_i$ for $i=0,1$.\qed
\end{itemize}
\end{teo}

Our next result shows that conditions (a) and (b$'$) characterize effective descent morphisms in $\Ord//X$ among those $p\colon E\to B$ with $p_x\colon E_x\to B_x$ surjective for every $x\in X$.

\begin{teo}\label{teo:53}
Let $X$ be a locally complete preordered set with bottom element. A morphism $p\colon E\to B$ in $\Ord//X$ such that $p_x\colon E_x\to B_x$ is surjective for every $x\in X$ is effective for descent in $\Ord//X$ if and only if:
\begin{itemize}
	\item [(a)] for each $b_2\leqslant b_1\leqslant b_0$ in $B$, there exist $e_2\leqslant e_1\leqslant e_0$ in $E$ with $p(e_i)=b_i$ for each $i=0,1,2$;
	\item [(b$'$)] for each $x\in X$ and each $b_1\leqslant b_0$ with $x\leqslant \beta(b_1)$, there exist $e_1\leqslant e_0$ with $x\leqslant \varepsilon(e_1)$ and $p(e_i)=b_i$ for $i=0,1$.\qed
\end{itemize}
\end{teo}

\begin{proof}
We only need to prove the necessity of conditions (a) and (b$'$).
Its proof follows directly the ``only if" proof of Theorem \ref{teo:OrdX}. Let $\bot$ be the bottom element of $X$. As in the proof of Theorem \ref{teo:OrdX} we apply Corollary \ref{cor:eff} to the functors
\[\xymatrix{\Ord//X\ar@<1ex>[rr]^-{S_1}&&\Ord\ar@<1ex>[ll]^-{J_1}}\]
where $S_1$ is the forgetful functor and $J_1$ assigns to each preordered set $A$ the pair $(A,\bot)$, with $\bot(a)=\bot$ for every $a\in A$, and conclude that an effective descent morphism in $\Ord//X$ is in particular effective for descent in $\Ord$, that is, it satisfies condition (a).

To show the necessity of (b$'$), let $x\in X$ and $b_1\leqslant b_0$ in $B$ with $x\leqslant \beta(b_1)$. Given a pullback diagram of $p$ and $f$ in $\mathcal{C}_0\times_{\mathcal{C}_1}\mathcal{C}_2$ with $A$ and $A_x$ as described in the proof of \ref{teo:OrdX}, both $A$ and $E\times_BA$ do not belong to $\Ord_X$ and therefore they do not belong to $\Ord//X$.
\end{proof}

There is another convenient way to express the surjectivity of each $p_x\colon E_x\to B_x$:

\begin{prop}
	The following conditions on a morphism $p:E\to B$ in $\Ord//X$ are equivalent:
	\begin{itemize}
		\item [(i)] $p_x\colon E_x\to B_x$ is surjective for every $x\in X$;
		\item [(ii)] for every $b\in B$ there exists $e\in E$ with $p(e)=b$ and $\varepsilon(e)\sim\beta(b)$.
	\end{itemize}
\end{prop}
\begin{proof}
	(ii)$\Rightarrow$(i) is obvious and so we only need to prove (i)$\Rightarrow$(ii). Suppose (i) holds. Given $b\in B$, since $b\in B_{\beta(b)}$, there exists $e\in E_{\beta(b)}$ with $p(e)=b$. Then we have $\beta(b)\leqslant\varepsilon(e)$ since $e$ belongs to $E_{\beta(b)}$, and $\varepsilon(e)\leqslant\beta p(e)=\beta(b)$ since $p$ is a morphism in $\Ord//X$.
\end{proof}

\begin{lem}\label{lem:55}
	Suppose that, for every $x\in X$, every subset of $\{x'\in X\mid x'\leqslant x\}$ has a largest element. If $p\colon E\to B$ is a pullback stable extremal epimorphism in $\Ord//X$, then $p_x\colon E_x\to B_x$ is surjective for every $x\in X$.
\end{lem}
\begin{proof}
	Given $b\in B$, consider the pullback diagram \[\xymatrix{p^{-1}(b)\ar[d]\ar[r]&\{b\}\ar[d]\\E\ar[r]_p&B}\] in $\Ord//X$, where the vertical arrow are the inclusion maps, and $\{b\}=(\{b\},\beta')$ with $\beta'(b)=\beta(b)$; accordingly $p^{-1}(b)=(p^{-1}(b),\varepsilon')$ with $\varepsilon'(e)=\varepsilon(e)$ for each $e\in p^{-1}(b)$. We have $\varepsilon(e)\leqslant\beta(b)$ for each $e\in p^{-1}(b)$, and we define $\beta'':\{b\}\to X$ by taking $\beta''(b)$ to be a largest element in the set $\{\varepsilon(e)\mid e\in p^{-1}(b)\}$. Then $(p^{-1}(b),\varepsilon')\to(\{b\},\beta')$, which is the top arrow in our pullback diagram, factors through $(\{b\},\beta'')\to(\{b\},\beta')$. Since $p$ is a pullback stable extremal epimorphism, it follows that $(\{b\},\beta'')\to(\{b\},\beta')$ is an isomorphism. Therefore $\beta''(b)\sim\beta'(b)$, which implies the existence of $e\in p^{-1}(b)$ with $\varepsilon(e)\sim\beta(b)$ and so completes the proof.
\end{proof}

From Theorem \ref{teo:53} and Lemma \ref{lem:55}, we immediately obtain:

\begin{teo}\label{teo:to}
Suppose that, for every $x\in X$, every subset of $\{x'\in X\mid x'\leqslant x\}$ has a largest element. Then a morphism $p\colon E\to B$ is effective for descent in $\Ord//X$ if and only if:
\begin{itemize}
	\item [(a)] for each $b_2\leqslant b_1\leqslant b_0$ in $B$, there exist $e_2\leqslant e_1\leqslant e_0$ in $E$ with $p(e_i)=b_i$ for each $i=0,1,2$;
	\item [(b$'$)] for each $x\in X$ and each $b_1\leqslant b_0$ with $x\leqslant \beta(b_1)$, there exist $e_1\leqslant e_0$ with $x\leqslant \varepsilon(e_1)$ and $p(e_i)=b_i$ for $i=0,1$.\qed
\end{itemize}
\end{teo}

\begin{remark}
	Although Theorem \ref{teo:53} is stronger than Theorem 3.9 in \cite{[CL2023]}, we still do not know how far it is from a complete characterization of effective descent morphisms in $\mathsf{Ord}//X$. Theorem \ref{teo:to} answers this question, but only under a strong additional condition on $X$.
\end{remark}

\end{document}